\documentclass[12pt]{amsart}

\usepackage{amsthm}
\usepackage{amssymb}
\usepackage{amsmath}
\usepackage{mathrsfs}
\numberwithin{equation}{section}
\newcommand{\bea}{\begin{eqnarray}}
\newcommand{\eea}{\end{eqnarray}}
\newcommand{\be}{\begin{eqnarray*}}
\newcommand{\ee}{\end{eqnarray*}}
\newtheorem{theorem}{Theorem}[section]

\newtheorem{definition}{Definition}[section]

\newtheorem{example}{Example}[section]
\newtheorem{remark}{Remark}[section]

\begin{document}
\title[Factoring Pseudoidentity Matrix Pairs]{Factoring Pseudoidentity Matrix Pairs}
%\author{Draft}
\author{Florian M. Sebert and Yi Ming Zou}
\thanks{F. Sebert is with d-fine consulting in Frankfurt, email: fmsebert@gmail.com. Y. M. Zou is with the 
Department of Mathematical Sciences, University of Wisconsin, 
Milwaukee, WI 53201, USA, email: ymzou@uwm.edu.}
%\date{September 25, 2008}

\maketitle

\begin{abstract}
The problem of factorization and parametrization of compactly supported biorthogonal wavelets was reduced to that of pseudoidentity matrix pairs by Resnikoff, Tian, and Wells in their 2001 paper. Based on a conjecture on the pseudoidentity matrix pairs of rank 2 stated in the same paper, they proved a theorem which gives a complete factorization result for rank 2 compactly supported biorthogonal wavelets. In this paper, we first provide examples to show that the conjecture is not true, then we prove a factorization theorem for pseudoidentity matrix pairs of rank $m\ge 2$. As a consequence, our result shows that a slightly modified version of the factorization theorem in the rank 2 case given by Resnikoff, Tian, and Wells holds. We also provide a concrete constructive method for the rank 2 case which is determined by applying the Euclidean algorithm to two polynomials.
\end{abstract}
%%%%%%%%%%%%%%%%%%%%%%%%%%%%%%%%%%%%%%%%%%%%%%%%%%%%%%%%%%%%%%%%%%%%%%%%%%%%%%%%%%%%%
%%%%%%%%%%%%%%%%%%%%%%%%%%%%%%%%%%%%%%%%%%%%%%%%%%%%%%%%%%%%%%%%%%%%%%%%%%%%%%%%%%%%%%%
\section{Introduction}
It is well-known that compactly supported orthogonal wavelets cannot be symmetric except for Haar wavelets \cite{Da, RTW, SN}. In signal processing, the symmetry property corresponds to the linear phase condition, and since biorthogonal wavelets permit symmetry, they are preferred over the orthogonal wavelets in applications such as digital filter banks. Because of the flexibility of the biorthogonal condition \cite{BS, CZ, CDF, Da, DS, HV, KHF, SN}, one can add other conditions to the definition of a biorthogonal wavelet according to the needs. To guarantee the existence of scaling functions and wavelet functions, a zeroth-order vanishing moment condition was imposed on the biorthogonal wavelet matrix in addition to the usual quadratic condition in \cite{RTW}. This extra requirement must be taken into account when analyzing the structure of these wavelets. In particular, the factorization results from the earlier investigations \cite{BS, CZ, DS, KHF, Po, RW} do not apply to these wavelets. It was proved in \cite{RTW} that any biorthogonal wavelet matrix pair can be decomposed into four components: an orthogonal component, a pseudoidentity matrix pair (defined in \cite{RTW}), an invertible matrix, and a constant matrix. Since the factorization and parametrization of the orthogonal wavelet matrices is well understood \cite{RTW, RW, SN}, the factorization of a biorthogonal wavelet matrix pair reduces to that of a pseudoidentity matrix pair. To obtain a complete factorization and therefore a parametrization of the biorthogonal wavelets, a conjecture was proposed for the pseudoidentity matrix pairs of rank $2$ in \cite{RTW}, and based on this conjecture, a complete factorization theorem was proved for the rank 2 case therein.
\par
In this paper, we prove a factorization theorem for a pseudoidentity matrix pair of rank $m\ge 2$, which in particular implies that the factorization theorem for the rank $2$ case stated in \cite{RTW} is true (with a slightly modified statement) in spite of the invalidness of the conjecture proposed there. We first provide examples to show that the conjecture is not true and then proceed to prove a factorization theorem. Our proof is elementary and constructive: it applies to any rank and provides an algorithm for the decomposition. 
\par
We organize our presentation as follows. In section 2, we first recall the definition and the basic properties of biorthogonal wavelets as given in \cite{RTW}, then we discuss some counterexamples to the conjecture stated in \cite{RTW}. In Section 3, we prove our factorization result on pseudoidentity matrix pairs and thus obtain a complete factorization for the biorthogonal wavelet matrix pairs. In section 4, we consider the rank 2 case in some detail and show that in this case, the factorization of a pseudoidentity matrix pair is basically determined by the Euclidean algorithm for finding the greatest common divisor (GCD) of two polynomials. We conclude the paper in section 5 with some discussion. 
\par 
\medskip
%%%%%%%%%%%%%%%%%%%%%%%%%%%%%%%%%%%%%%%%%%%%%%%%%%%%%%%%%%%%%%%%%%%%%%%%%%%%%%%%%%%%%
%%%%%%%%%%%%%%%%%%%%%%%%%%%%%%%%%%%%%%%%%%%%%%%%%%%%%%%%%%%%%%%%%%%%%%%%%%%%%%%%%%%%%%% 
\section{Biorthogonal wavelets} 
\par
Let $\mathbb{C}$ be the set of complex numbers, let $\mathbb{F}\subseteq\mathbb{C}$ be a subfield invariant under the complex conjugation, and let $R=\mathbb{F}[z,z^{-1}]$ be the ring of Laurent polynomials over $\mathbb{F}$. The reason for us to work with $\mathbb{F}$ instead of just $\mathbb{C}$ is that it will allow more flexibility in the applications, since the discussions here apply to a whole range of fields including the rational numbers $\mathbb{Q}$, the real numbers $\mathbb{R}$, the complex numbers $\mathbb{C}$, as well as other algebraic extensions of $\mathbb{Q}$. For instance, in some applications of digital filters, it is desirable to work with rational numbers, since it is easier to implement the filters. We denote by $gl(m,R)$ the set of all $m\times m$ matrices over $R$ and denote by $GL(m,R)$ the set of all invertible $m\times m$ matrices over $R$. We also denote by $gl(m,\mathbb{F})$ and $GL(m,\mathbb{F})$ the set of $m\times m$ matrices and the set of invertible $m\times m$ matrices over $\mathbb{F}$, respectively. For a matrix $A(z)\in gl(m,R)$, we can write
\bea\label{Laurent1}
A(z)=\sum_{k\in\mathbb{Z}}A_kz^{-k}=\sum_{k=k_0}^{k_1}A_kz^{-k},
\eea  
where $A_k\in gl(m,\mathbb{F}), k_0\le k\le k_1$, and $k_0$ and $k_1$ are the smallest and the largest indices such that $A_k\ne 0$, respectively. We call the expression of $A(z)$ in (\ref{Laurent1}) the Laurent series of $A(z)$. The number $g=k_1-k_0+1$ is called the genus of $A(z)$. If $A(z)$ has genus $g$, then we can identify $A(z)$ with the block matrix
\bea\label{blk}
A:=(A_{k_0}, \ldots, A_0, \ldots, A_{k_1} ),
\eea
where
\bea
A_k=(a_{i,km+j}), \quad 0\le i\le m-1, 0\le j\le m-1, k_0\le k\le k_1,
\eea
and say that $A(z)$ has size $m\times mg$. We define the adjoint $\tilde{A}(z)$ of $A(z)$ by
\bea
\tilde{A}(z)=A^{\ast}(z^{-1})=\sum_{k=k_0}^{k_1}A^{\ast}_kz^{k}=\sum_{k=-k_1}^{-k_0}A^{\ast}_{-k}z^{-k},
\eea 
where $A_k^{\ast}$ denotes the conjugate transpose of the complex matrix $A_k$.
\par
We are now ready to recall the definition of a biorthogonal wavelet matrix pair and the definition of a pseudoidentity matrix pair from \cite{RTW}.
\begin{definition}
A pair of $m\times mg$ matrices $(L =(l_{i,j}), R =(r_{i,j}))$ (as in (\ref{blk})) is said to be a biorthogonal wavelet matrix pair of rank $m$ and genus $g$ if the following conditions are satisfied:
\bea\label{qua}
L(z)\tilde{R}(z)=mI_m,
\eea 
and 
\bea\label{lin}
\sum_{j}l_{i,j}=\sum_{j}r_{i,j}=\left\{\begin{array}{ccl}
                                          m, & \mbox{if} & i=0,\\
                                          0, & \mbox{if} & 1\le i\le m-1,   
                                      \end{array}\right.
\eea
where $I_m$ is the $m\times m$ identity matrix.
\end{definition}
\par
In the literature, the matrix $L(z)$ is called the \textit{analysis matrix} and the matrix $R(z)$ is called the \textit{synthesis matrix} of the biorthogonal wavelet pair. Condition (\ref{qua}) is called the \textit{quadratic} or \textit{perfect reconstruction} condition and condition (\ref{lin}) is called the \textit{linear condition}. Note that (\ref{lin}), which is also referred to as the \textit{zeroth-order vanishing moment condition}, is a necessary condition for the existence of scaling functions and wavelet functions. This is one of the main differences between wavelets and perfect reconstruction filter banks.  
\begin{definition}
A matrix pair $(C(z), D(z))$ is called a pseudoidentity matrix pair if 
\bea\label{pimp}
\begin{array}{ccc}
\displaystyle C(z)=\sum_{k=0}^{k_c}C_kz^{-k}, \quad & \displaystyle D(z)=\sum_{k=k_d}^{0}D_kz^{-k}, & \; (k_d\le 0\le k_c)\\
{} \quad &{}\\
C(z)\tilde{D}(z)=I_m, \quad &C(1)=D(1)=I_m. &{}
\end{array}
\eea
\end{definition}
\par\medskip
We remark that it follows from the definition that 
\bea\label{det}
\det(C(z))=\det(D(z)) =1. 
\eea
For some other properties of the pseudo identity matrix pairs, please see \cite{RTW}.
\par\medskip
Recall that a primitive paraunitary matrix is an $m\times 2m$ matrix of the form
\bea
V(z)=I_m-vv^{\ast}+vv^{\ast}z^{-1},
\eea
where $v\in\mathbb{F}^m$ is a unit column vector. Note that a primitive paraunitary matrix satisfies $V(z)\tilde{V}(z) = I_m$
\par\medskip
The following theorem, due to Resnikoff, Tian, and Wells, reduces the factorization of a biorthogonal wavelet matrix pair to that of a pseudoidentity matrix pair.
\begin{theorem} (see [11, Thm. 3.4]) A pair of $m\times mg$ matrices $(L,R)$ is a biorthogonal wavelet pair of rank $m$ if and only if there exist primitive paraunitary matrices $V_1,\ldots, V_d$, $d\ge 0$, such that 
\bea\begin{array}{c}
L(z)=z^{-k_0}V_1(z)\cdots V_d(z)C(z)\left(\begin{array}{cc} 1 & 0\\ 0 & G
                                    \end{array}\right)\mathbf{H},\\
                                    {}\\
R(z)=z^{-k_0}V_1(z)\cdots V_d(z)D(z)\left(\begin{array}{cc} 1 & 0\\ 0 & (G^{-1})^{\ast}
                                    \end{array}\right)\mathbf{H},                                    

\end{array}
\eea
where $k_0\in\mathbb{Z}$, $d=b-mk_0$, $b$ is the exponent of $\det L(z)$, $G\in GL(m-1,\mathbb{F})$, $\mathbf{H}$ is the canonical Haar matrix of rank $m$, and $(C(z), D(z))$ is a pseudoidentity matrix pair such that the genus of $C(z)$ is $\le g$.
\end{theorem}
\par\medskip
We consider an example of pseudoidentity matrix pairs.
\begin{example}
For any sequence of nonzero numbers 
\bea
a_1, a_2,\ldots, a_k\in \mathbb{F}, 
\eea
let $0<m_1<\cdots<m_k$ be $k$ positive integers, let
\bea
a_0 =\sum_{i=1}^ka_i,\qquad u(z)=\sum_{i=1}^ka_iz^{-m_i},
\eea
let
\bea
C(z)=\left(\begin{array}{cc} 1-a_0 & a_0 \\
                             -a_0  & 1+a_0 \end{array}\right) 
                             + \sum_{i=1}^{k}\left(\begin{array}{cc} a_i & -a_i \\
                             a_i  & -a_i \end{array}\right)z^{-m_i},
\eea
and let 
\bea
D(z)=\left(\begin{array}{cc} 1+\bar{a}_0 & \bar{a}_0 \\
                             -\bar{a}_0  & 1-\bar{a}_0 \end{array}\right) 
                             + \sum_{i=1}^{k}\left(\begin{array}{cc} -\bar{a}_i & -\bar{a}_i \\
                             \bar{a}_i  & \bar{a}_i \end{array}\right)z^{m_i},
\eea
where $\bar{a}_i$ is the complex conjugate of $a_i$. Then 
\bea\begin{array}{c}
 C(1)=D(1)=I_2, \quad C(z)\tilde{D}(z) = I_2,\\
 {}\\
  \det(C(z))=\det(D(z))=1.
  \end{array}
\eea
The first equality follows from the definition, and the second and the third equalities can be easily seen since
\bea
C(z)=\left(\begin{array}{cc} 1-a_0+u(z) & a_0-u(z) \\
                             -a_0+u(z)  & 1+a_0-u(z) \end{array}\right)
\eea
has inverse
\bea
(C(z))^{-1}=\left(\begin{array}{cc} 1+a_0-u(z) & -a_0+u(z) \\
                             a_0-u(z)  & 1-a_0+u(z) \end{array}\right),
\eea
and since $\tilde{D}(z)=(C(z))^{-1}$.
\end{example}
\par\medskip
For instance, if we take $k=2$, $a_1=a_2 =1$, and $m_1=1, m_2=2$, then
\bea\label{pseu2}
C(z)= \left(\begin{array}{cc} 
             -1 & 2\\
             -2 & 3
            \end{array}\right)
            + \left(\begin{array}{cc} 
             1 & -1\\
             1 & -1
            \end{array}\right)z^{-1} 
            + \left(\begin{array}{cc} 
             1 & -1\\
             1 & -1
            \end{array}\right)z^{-2},
\eea
and
\bea
D(z)= \left(\begin{array}{cc} 
             3 & 2\\
             -2 & -1
            \end{array}\right)
            + \left(\begin{array}{cc} 
             -1 & -1\\
             1 & 1
            \end{array}\right)z^{1} 
            + \left(\begin{array}{cc} 
             -1 & -1\\
             1 & 1
            \end{array}\right)z^{2}.
\eea
\par\medskip
\begin{remark} Since 
\bea 
\left(\begin{array}{cc} a_i & -a_i \\
                             a_i  & -a_i \end{array}\right)^2 =0,
\eea 
these pseudoidentity matrix pairs provide counterexamples to the following conjecture (Conjecture 1) in \cite{RTW}:
\par\medskip
{\it Conjecture}: For $m=2$ and $C(z)$ and $D(z)$ as in (\ref{pimp}), if $C(z)\tilde{D}(z) = I_2$ and $p$ is the smallest positive integer such that $C_{k_c-p}\ne 0$, then both $C_{k_c-p}$ and $D_{k_d+p}$ must be invertible matrices.
\end{remark}
\par\medskip
\begin{remark}
A weaker condition than the condition stated in the conjecture that guarantees the factorization of a pseudoidentity matrix pair was also given in \cite{RTW} (see the remark after the proof of Thm. 4.4 in \cite{RTW}). We note that our examples do not satisfy this weaker condition either, since for the matrices in (\ref{pseu2}) the condition
\be
-C_{k_c-p}N + C_{k_c} = 0
\ee
becomes
\be
-\left(\begin{array}{cc} 
             1 & -1\\
             1 & -1
            \end{array}\right)N
            + \left(\begin{array}{cc} 
             1 & -1\\
             1 & -1
            \end{array}\right) = \left(\begin{array}{cc} 
             1 & -1\\
             1 & -1
            \end{array}\right)(I_2-N) =0. 
\ee
But this last equality cannot be true, since if $N$ is such that $N^2=0$, then $I_2-N$ is invertible with $(I_2-N)^{-1}=I_2+N$. 
\end{remark}
\begin{remark}
We note that the matrix in (\ref{pseu2}) can be factored as
\bea\begin{array}{ll}\label{rem3}
C(z) &= \left(\left(\begin{array}{cc} 
             1 & 0\\
             0 & 1
            \end{array}\right)
            - \left(\begin{array}{cc} 
             1 & -1\\
             1 & -1
            \end{array}\right)
            + \left(\begin{array}{cc} 
             1 & -1\\
             1 & -1
            \end{array}\right)z^{-1}\right)\\ 
            {}&{}\\
         {}&   \cdot\left(\left(\begin{array}{cc} 
             1 & 0\\
             0 & 1
            \end{array}\right)
            - \left(\begin{array}{cc} 
             1 & -1\\
             1 & -1
            \end{array}\right) 
            + \left(\begin{array}{cc} 
             1 & -1\\
             1 & -1
            \end{array}\right)z^{-2}\right).
    \end{array}        
\eea
Note also that the degrees do not add up.
\end{remark}
\par\medskip
\section{Complete factorization}
\par
We now prove the following theorem.
\begin{theorem}
Let $(C(z),D(z))$ be a pair of matrices of rank $m\ge 2$ defined by 
\bea
\displaystyle C(z)=\sum_{k=0}^{k_c}C_kz^{-k}, \quad & \displaystyle D(z)=\sum_{k=k_d}^{0}D_kz^{-k}, \quad k_c,k_d>0.
\eea
Then $(C(z),D(z))$ is a pseudoidentity matrix pair if and only if there exist a positive integer $r$ and nilpotent matrices
\be
N_i\in gl(m,\mathbb{F}),\quad N_i^2=0,\quad 1\le i\le r,
\ee
such that
\be
 C(z)=L_{N_r}(z)\cdots L_{N_2}(z)L_{N_1}(z),\quad D(z)=R_{N_r}(z)\cdots R_{N_2}(z)R_{N_1}(z),
\ee
with
\be
L_{N_i}(z) = I_m-N_i+N_iz^{-k_i},\quad R_{N_i}(z) = I_m+N_i^{\ast}-N_i^{\ast}z^{k_i}, \quad 1\le i\le r.
\ee
\end{theorem}
\begin{proof} It is clear that the conditions are sufficient. To prove the necessary part, note that since $D(z)$ is just the adjoint of the inverse of $C(z)$, we need only to prove the statement for $C(z)$. To avoid writing the negative exponents, we make a change of variable (not essential) by letting $t=z^{-1}$. We abuse notation and write
\bea
C(t)=\sum_{k=0}^{k_c}C_kt^k.
\eea
\par
Since the entries of $C(t)$ are elements of the Euclidean domain $\mathbb{F}[t]$, by the Division algorithm, we can use elementary row and column operations of the form (type I elementary operations): 
\begin{center} \textit{add a multiple of a row or column to another row or column}, \end{center}
to reduce $C(t)$ to a diagonal matrix
\bea
C^{\prime}(t)=\left(\begin{array}{cccc} 
d_1 & {} & {} & {}\\
{} & d_{2} & {} & {}\\
{} & {} & \ddots & {}\\
{} & {} & {} & d_{m} \end{array}\right),
\eea
where 
\bea
0\ne d_i\in \mathbb{F}[t], \quad 1\le i\le m.
\eea
This is possible since: (a) we can use elementary row and column operations to reduce $C(t)$ to the given form (see for example [1, pp. 459-460]--here, the statement is actually weaker, since we do not require that $d_i/d_{i-1}$); (b) the operations that interchange two rows or columns can be obtained from type I elementary operations together with scalar multiplications by nonzero numbers to rows and columns; and, (c) in our diagonal form, we allow the elements $d_i, 1\le i\le m$, to be defined up to nonzero scalar multiples (e.g. we do not care about negative signs), and thus multiplication by nonzero numbers to rows and columns are not needed here.
\par 
Writing our reduction process as matrix multiplications, we have
\bea\label{ct1}
C^{\prime}:=C^{\prime}(t)=P_u(t)\cdots P_1(t)C(t)P^{\prime}_s(t)\cdots P^{\prime}_1(t),
\eea
where the $P_k(t), 1\le k\le u$, and $P^{\prime}_l(t),1\le l\le s$, are elementary matrices of the form:
\bea\label{e1}
I_m+f(t)E_{ij},\quad i\ne j,\quad f(t)\in\mathbb{F}[t],
\eea
where $E_{ij}$ is the $m\times m$ matrix with 1 at the $ij$th entry and 0 elsewhere. 
\par
Since $\det(C(t))=1$ (see (\ref{det})), by taking determinants on both sides of (\ref{ct1}), we conclude that all $d_i\in\mathbb{F}, 1\le i\le m$, that is, $C^{\prime}$ is a constant diagonal matrix.
\par
Let 
\bea
E_k(t)= \left\{\begin{array}{lcl} (P_k(t))^{-1}, & \mbox{if} &1\le k\le u,\\
C^{\prime}(P^{\prime}_{k-u}(t))^{-1}(C^{\prime})^{-1},&\mbox{if} & u+1\le k\le u+s.
\end{array}\right.
\eea
Then all $E_k(t)$ are elementary matrices of the form (\ref{e1}), and 
\bea\label{ct2}
C(t)=E_1(t)\cdots E_{u+s}(t)C^{\prime}.
\eea
\par
Consider elementary matrices of the form (\ref{e1}). Let 
\bea
E(t)=I_m+f(t)E_{ij}, \quad i\ne j, 
\eea
with
\bea
f(t)=a_nt^n+a_{n-1}t^{n-1}+\cdots+a_0 \in\mathbb{F}[t].
\eea
We can assume that $f(t)\ne 0$, since $f(t)=0$ gives the identity matrix, and factor $E(t)$ as
\bea\label{et1}
\qquad E(t)= (I_m+(a_nt^n-a_n)E_{ij})\cdots (I_m+(a_1t-a_1)E_{ij})(I_m+b_0E_{ij}),
\eea
where a term exists only if the corresponding coefficient $a_k$ is not 0, and 
\bea
b_0=\sum_{i=0}^{n}a_i.
\eea
\par
Note that the factors in (\ref{et1}) involving $t$ can be written as
\bea\label{n1}
I_m -a_kE_{ij} +a_kE_{ij}t^k = I_m -N+Nt^k,
\eea
where $N:=a_kE_{ij}$ is nilpotent and $N^2=0$. If we conjugate the right hand side in (\ref{n1}) by $G\in GL(m,\mathbb{F})$, we have
\bea\label{n2}
G(I_m -N+Nt^k)G^{-1} = I_m -N^{\prime} + N^{\prime}t^k,
\eea
with $(N^{\prime})^2=(GNG^{-1})^2=0$. Therefore, upon plugging (\ref{et1}) and (\ref{n1}) into (\ref{ct2}) and moving all the constant matrices (e.g. the last term in (\ref{et1})) to the end, we obtain
\bea\label{ct3}
C(t)= L_{N_r}(t)\cdots L_{N_2}(t)L_{N_1}(t)QC^{\prime}\quad\mbox{for some $r>0$},
\eea
where the factors $L_{N_i}(t)$ are as defined in (\ref{n1}) or (\ref{n2}), and $Q\in GL(m,\mathbb{F})$. Setting $t=1$ in (\ref{ct3}) and noticing that $L_{N_i}(1)=I_m, 1\le i\le r$, we have $QC^{\prime}=I_m$. Substituting back $z^{-1}=t$, we have proved the theorem.
\end{proof}
\par
\begin{remark} Since the factors are formed by nilpotent matrices, we should not expect that the degrees of the factors add up to the degree of $C(z)$ as (\ref{rem3}) shows.
\end{remark}
\begin{remark}
In general, the factorization depends on the reduction process and it is not unique.
\end{remark}
\par\medskip
\begin{example} Here we consider a $3\times 3$ example. Let
\be
C(z) &=& \left(\begin{array}{ccc}
-z^{-2}-z^{-1}+3 & -z^{-2}-z^{-1}+2 & z^{-2}+z^{-1}-2\\
2z^{-2}+2z^{-1}-4 & 2z^{-2}+2z^{-1}-3 & -2z^{-2}-2z^{-1}+4\\
z^{-2}+z^{-1}-2 & z^{-2}+z^{-1}-2 & -z^{-2}-z^{-1}+3\\
\end{array}\right),\\
{} &{}&{}\\
E_1 &=& \left(\begin{array}{ccc}
1 &0 & 1\\
0 &1 &-2\\
0 &0 &1
\end{array}\right)
= \left(\begin{array}{ccc}
1 &0 & 1\\
0 &1 &0\\
0 &0 &1
\end{array}\right)\left(\begin{array}{ccc}
1 &0 & 0\\
0 &1 &-2\\
0 &0 &1
\end{array}\right),\\
{} &{}&{}\\
E_2 &=& \left(\begin{array}{ccc}
1 &0 & 0\\
0 &1 &0\\
-z^{-2}-z^{-1}+2 &0 &1
\end{array}\right),\\ 
{} &{}&{}\\
E_3 &=& \left(\begin{array}{ccc}
1 &0 & 0\\
0 &1 &0\\
0 &-z^{-2}-z^{-1}+2 &1
\end{array}\right),\\
{} &{}&{}\\
E_4 &=& \left(\begin{array}{ccc}
1 &0 &-1\\
0 &1 &2\\
0 &0 &1
\end{array}\right)
= \left(\begin{array}{ccc}
1 &0 &-1\\
0 &1 &0\\
0 &0 &1
\end{array}\right)\left(\begin{array}{ccc}
1 &0 & 0\\
0 &1 &2\\
0 &0 &1
\end{array}\right).
\ee
Then
\be
E_3E_2E_1C(z)E_4=I_4.
\ee
Thus after moving the $E_i$'s to the right hand side and simplifying, we obtain
\be
C(z)=(I_4-N_1+N_1z^{-1})(I_4-N_1+N_1z^{-2})(I_4-N_2+N_2z^{-1})(I_4-N_2+N_2z^{-2}),
\ee 
where
\be
N_1=\left(\begin{array}{ccc}
-1 &0 &-1\\
2 &0 &2\\
1 &0 &1
\end{array}\right),\quad N_2 =\left(\begin{array}{ccc}
0 &-1 &2\\
0 &2 &-4\\
0 &1 &-2
\end{array}\right).
\ee
\end{example}
As a consequence of \textbf{Thm. 3.1}, we have a refinement of \textbf{Thm. 2.1}:
\begin{theorem} (biorthogonal wavelet factorization theorem)
For $m\ge 2$, a pair of $m\times mg$ matrices $(L,R)$ is a biorthogonal wavelet pair of rank $m$ if and only if there exist primitive paraunitary matrices $V_1,\ldots, V_d$, $d\ge 0$, and nilpotent matrices
\be
N_i\in gl(m,\mathbb{F}),\quad N_i^2=0,\quad 1\le i\le r,
\ee
such that 
\bea\begin{array}{c}
L(z)=z^{-k_0}V_1(z)\cdots V_d(z)L_{N_r}(z)\cdots L_{N_1}(z)\left(\begin{array}{cc} 1 & 0\\ 0 & G
                                    \end{array}\right)\mathbf{H},\\
                                    {}\\
R(z)=z^{-k_0}V_1(z)\cdots V_d(z)R_{N_r}(z)\cdots R_{N_1}(z)\left(\begin{array}{cc} 1 & 0\\ 0 & (G^{-1})^{\ast}
                                    \end{array}\right)\mathbf{H},                                    

\end{array}
\eea
where $k_0\in\mathbb{Z}$, $d=b-mk_0$, $b$ is the exponent of $\det L(z)$, $G\in GL(m-1,\mathbb{F})$, $\mathbf{H}$ is the canonical Haar matrix of rank $m$, and $L_{N_i}, R_{N_i}, 1\le i\le r$, are as defined in \textbf{Thm. 3.1} such that the genus of 
\be
 C(z)=L_{N_r}(z)\cdots L_{N_1}(z)
\ee
is $\le g$.
\end{theorem}
\par\medskip
\section{The rank 2 case}
For a rank 2 pseudoidentity matrix pair $(C(z),D(z))$, the factorization can be described in a concrete way. We will see that the computational complexity in the factorization is basically determined by the operations involved in carrying out the Euclidean algorithm to determine the GCD of {\it two polynomials} with coefficients in $\mathbb{F}$. Again, to avoid writing the negative exponents, we use the variable $t$ instead of $z$ in our discussion as before.
Let $a=a(t),b=b(t),c=c(t),d=d(t)\in \mathbb{F}[t]$ be such that
\bea\label{ct41}
C(t)=\left(\begin{array}{cc}
a & b\\
c & d
\end{array}\right).
\eea
\par
Since $C(1)=I_2$, $ad\ne 0$, and $b(1)=c(1)=0$. If one of $b$ and $c$ is 0, then $\det(C(t))=1$ implies that $a$ and $d$ are constants and hence $a=d=1$. If, say $c=0$, then we can write 
\be
b(t)=(b_nt^n-b_n)+\ldots+(b_1t-b_1),
\ee
and factor $C(t)$ as 
\be
C(t)= L_{N_n}(t)\cdots L_{N_2}(t)L_{N_1}(t),
\ee
where
\be
L_{N_i}(t)=I_2-\left(\begin{array}{cc}
0 & b_i\\
0 & 0
\end{array}\right)+ \left(\begin{array}{cc}
0 & b_i\\
0 & 0
\end{array}\right)t^{i}, \quad 1\le i\le n.
\ee
\par
So assume $bc\ne 0$. To fix our case of discussion, let $\deg(a)\le\deg(c)$ and apply the division algorithm repeatedly to get
\bea\label{eu}\begin{array}{rcl}
c &=& q_1a + r_1,\quad \deg(r_1)<\deg(a),\\
a &=& q_2r_1 + r_2,\quad \deg(r_2)<\deg(r_1),\\
{} &\vdots &{}\\
r_{n-2} &=& q_nr_{n-1} + r_{n},\quad \deg(r_n)<\deg(r_{n-1}),\\
r_{n-1} &=& q_{n+1}r_n.
\end{array}
\eea 
Thus, by multiplying 
\bea\label{41}
(I_2 - q_{n+1}E_{uv})\cdots (I_2 - q_2E_{12})(I_2 - q_1E_{21}), 
\eea
to $C(t)$ from the left (where the index $uv=12$ or $21$ depends on the case), we obtain
\bea
\left(\begin{array}{cc}
r_n & b_1\\
0 & d_1
\end{array}\right)\quad\mbox{or}\quad \left(\begin{array}{cc}
0 & b_1\\
r_n & d_1
\end{array}\right).
\eea
Since the determinant of (\ref{41}) is 1, we must have $r_n,d_1$ (or $b_1$) $\in\mathbb{F}$. Multiplying 
\bea\label{42}
I_2-(d_1)^{-1}b_1E_{12}\quad\mbox{or} \quad I_2-(b_1)^{-1}d_1E_{21}
\eea
from the left again, we get a constant matrix:
\bea\label{43}
\left(\begin{array}{cc}
r_n & 0\\
0 & d_1
\end{array}\right)\quad\mbox{or}\quad \left(\begin{array}{cc}
0 & b_1\\
r_n & 0
\end{array}\right).
\eea
Thus, letting the matrix that was actually used in the process from (\ref{42}) be 
\be 
I_2-g(t)E_{rs},
\ee 
and the corresponding matrix from (\ref{43}) be $C^{\prime}$ 
we get
\bea\label{ct42}
\qquad C(t)=(I_2 + q_1E_{21})(I_2 + q_2E_{12})\cdots (I_2 + q_{n+1}E_{uv})(I_2+g(t)E_{rs})C^{\prime}.
\eea
Factoring each of the factors from the right hand side further as before if needed (see the proof of \textbf{Thm. 3.1}), we get the factorization desired. 
\par
To summarize our discussion, we introduce some notation. For an integer $m\ge 2$, a $2\times 2$ matrix $E_{ij}, i\ne j$, and a polynomial
\bea
f(t)=a_nt^n+\cdots+a_1t+a_0\in\mathbb{F}[t],
\eea
we define the factorization of $I_m+f(t)E_{ij}$ as follows. Let
\bea
\qquad L_{ij,a_k}=I_m-a_kE_{ij}+t^ka_kE_{ij},\quad 1\le k\le n,\quad E_{ij,f(1)}=I_m+f(1)E_{ij},
\eea
and set
\bea
L_{ij,f}(t)=L_{ij,a_n}\cdots L_{ij,a_1}.
\eea 
Then
\bea
 I_m+f(t)E_{ij}=L_{ij,f}(t)E_{ij,f(1)}.
\eea
\par
\begin{theorem}
Let $C(t)$ be defined as in (\ref{ct41}) and assume (\ref{eu}) and (\ref{ct42}) hold. Then $C(z)$ can be factored completely into a product of matrices of the form:
\be
I_2-N+z^{-k}N, \quad N^2=0,
\ee
by substituting $z^{-1}$ for $t$ in  
\be
L_{12,q_1}L^{\prime}_{21,q_2}\cdots L^{\prime}_{uv,q_{n+1}}L^{\prime}_{rs,g},
\ee
where $q_k, 1\le k\le n+1$, are as in (\ref{eu}); $g=g(t), uv, rs$, are as in (\ref{ct42}); and  
\be
L^{\prime}_{21,q_2} &=& E_{12,q_1(1)}L_{21,q_2}E_{12,-q_1(1)},\\
L^{\prime}_{12,q_3} &=& E_{12,q_1(1)}E_{21,q_2(1)}L_{21,q_2}E_{21,-q_2(1)}E_{12,-q_1(1)},\\
{} &\vdots& {}
\ee
are the conjugates of the corresponding $L_{ij,f}$.
\end{theorem}
\begin{proof} See the last paragraph of the proof of \textbf{Thm. 3.1}.
\end{proof}
%%%%%%%%%%%%%%%%%%%%%%%%%%%%%%%%%%%%%%%%%%%%%%%%%%%%%%%%%%%%%%%%%%%%%%%%%%%%%%%%%%%%%
%%%%%%%%%%%%%%%%%%%%%%%%%%%%%%%%%%%%%%%%%%%%%%%%%%%%%%%%%%%%%%%%%%%%%%%%%%%%%%%%%%%%%%%
\section{Conclusion and Discussion}
\par
The work of Resnikoff, Tian, and Wells \cite{RTW} shows that any biorthogonal wavelet matrix pair can be decomposed into four components: an orthogonal component, a pseudoidentity matrix pair, an invertible matrix, and a constant matrix. Their work reduced the parametrization and factorization of a biorthogonal wavelet matrix to that of a pseudoidentity matrix pair. The main contribution of the current paper is the factorization theorem for a pseudoidentity matrix pair of arbitrary rank proved in section 3. Our result implies in particularly in the case $m=2$, that a slightly modified version of the factorization theorem stated in \cite{RTW} is true in spite of the invalidness of the reduction procedure proposed there. We also provided several examples to explain our results. As the examples show, since the factorizations involve nilpotent matrices, in contrast to the orthogonal case \cite{Po}, there is no uniqueness result even for the rank 2 case. Furthermore, in a factorization of a pseudoidentity pair, the degrees of the factors do not add up to the genus of the pair in general for the same reason. Hence the parametrization of these wavelets is not immediately clear from the factorization theorem. 
\par
To explain this point in more detail, recall that \cite{RW, SN} in the orthogonal case, the factorization of paraunitary matrices into products of primitive paraunitary matrices provides an \textit{onto} map from 
\be
\underbrace{S^{2m-1}\times S^{2m-1}\times\cdots\times S^{2m-1}}_{\mbox{{\footnotesize $g$ factors}}}\times U(m) 
\ee
to the set of all paraunitary matrices of genus $g$,
where $U(m)$ is the set of unitary matrices (constant) and 
\be
S^{2m-1}=\{v\in\mathbb{F}^m | v^{\ast}v=1\}.
\ee
But in the biorthogonal case, the factorization does not lead to a similar result. 
More precisely, if $(C(z),D(z))$ is a pseudoidentity matrix pair of genus $g$, and
\be\label{c51}
 C(z)=L_{N_r}(z)\cdots L_{N_2}(z)L_{N_1}(z),\quad D(z)=R_{N_r}(z)\cdots R_{N_2}(z)R_{N_1}(z),
\ee
where the $N_i, 1\le i\le r$, are nilpotent matrices such that $N_i^2=0$, and
\be
L_{N_i}(z) = I_m-N_i+N_iz^{-k_i},\quad R_{N_i}(z) = I_m+N_i^{\ast}-N_i^{\ast}z^{k_i}, \quad 1\le i\le r,
\ee
then we only have 
\be\label{c52}
k_1+k_2+\cdots+k_r\ge g,
\ee 
and for some factorizations, strict inequality can hold.
Thus, the parametrization problem in the biorthogonal case is more complicated and needs further study, and the related geometry problems could offer solutions. These topics are currently under investigation. 
\par
Finally, we would like to point out that though the factorization of a pseudoidentity matrix pair provided in \textbf{Thm. 3.1} is useful in analyzing the structure of the biorthogonal wavelets such as parametrization, in applications, one may want to use the factorization provided by (\ref{ct2}) and (\ref{et1}) directly. We state this alternative factorization as a theorem below. Note that because of the mix of the constant and nonconstant matrices in the factorization, the statement is not as neat theoretically. Let $\delta_{ij}$ be $1$ or $0$ according to whether $i=j$ or not.
\begin{theorem} (lattice structure) The matrix $C(z)$ of size $m\times gm$ is a pseudo identity matrix if and only if 
\bea
C(z)=E_r(z)\cdots E_{1}(z)C^{\prime},
\eea
where the $E_s(z), 1\le s\le r$, are elementary matrices of the form
\bea
\qquad I_m +(\delta_{k_s,0}-1)a_sE_{ij} +a_sE_{ij}z^{-k_s},\quad i\ne j, \quad a_s\ne 0, \quad 0\le k_s\le g,
\eea
and $C^{\prime}$ is a constant diagonal matrix such that $\det (C^{\prime})=1$ and 
\bea
\left(\prod_{\substack{1\le s\le r\\ k_s=0}}{E_s(z)}\right)C^{\prime}=I_m.
\eea
\end{theorem}  
\par 
\medskip
%%%%%%%%%%%%%%%%%%%%%%%%%%%%%%%%%%%%%%%%%%%%%%%%%%%%%%%%%%%%%%%%%%%%%%%%%%%%%%%%
%%%%%%%%%%%%%%%%%%%%%%%%%%%%%%%%%%%%%%%%%%%%%%%%%%%%%%%%%%%%%%%%%%%%%%%%%%%%%%%%

\end{document}